\documentclass[reqno]{amsart}
\usepackage{latexsym,fancyhdr}
\usepackage{amsmath,amssymb,xcolor}
\usepackage{multicol}
\usepackage{mathtools}
\usepackage{times}
\usepackage{pb-diagram}
\usepackage[all,cmtip]{xy}
\usepackage{lineno}
\usepackage{caption}

\usepackage{algorithm}
\usepackage[noend]{algpseudocode}

\newtheorem{Tma}{Theorem}[section]

\newtheorem{lemma}[Tma]{Lemma}
\newtheorem{conjecture}{Conjecture}[section]
\newtheorem{proposition}[Tma]{Proposition}

\newtheorem{theorem}[Tma]{Theorem}

\theoremstyle{definition}

\newtheorem{remark}[Tma]{Remark}

\newcommand{\ext}{\mathcal{E}}
\newcommand{\HM}{\mathcal{H}}
\newcommand{\diag}{{\rm diag}}
\newcommand{\Tri}{\tau}
\newcommand{\Inv}{\iota}

\newcommand{\LaRa}{\langle\pm \rangle}
\newcommand{\myem}[1]{{\bfseries \textit{#1}}}  

\def\figa{(i)}
\def\figb{(ii)}

\newcounter{ithmcount}
\newenvironment{iprf}{\begin{list}{{\rm
	\alph{ithmcount})}}{\usecounter{ithmcount}\labelwidth-5pt
      \leftmargin0pt \topsep3pt \itemsep1pt \parsep2pt}}{\end{list}}

\newenvironment{ithm}{\begin{list}{{\rm \alph{ithmcount})}}{\usecounter{ithmcount}\labelwidth18pt
      \leftmargin18
pt \topsep3pt \itemsep1pt \parsep2pt}}{\end{list}}

\vspace*{-3.5cm}

\title{Constructing cocyclic Hadamard matrices of order $4p$}
\author{Santiago Barrera Acevedo}
\address{Monash University, School of Mathematical Sciences, Clayton 3800 VIC, Australia}
\email{Santiago.Barrera.Acevedo@monash.edu}

\author{Padraig \'O Cath\'ain}
\address{Worcester Polytechnic Institute, Mathematical Sciences Department, Worcester, MA, USA}
\email{pocathain@wpi.edu}

\author{Heiko Dietrich}
\address{Monash University, School of Mathematical Sciences, Clayton 3800 VIC, Australia}
\email{heiko.dietrich@monash.edu}

\renewcommand{\leq}{\leqslant}

\begin{document}

\reversemarginpar

\begin{abstract}
  \noindent
  Cocyclic Hadamard matrices (CHMs) were introduced by de Launey and Horadam as a class of Hadamard matrices  with interesting algebraic properties. \'O Cath\'ain and R\"oder described a classification algorithm for CHMs of order $4n$ based on relative difference sets in groups of order $8n$; this led to the classification of all CHMs of order at most 36. Based on work of de Launey and Flannery, we describe a classification algorithm for CHMs of order $4p$ with $p$ a prime; we prove refined structure results and provide a classification for $p \leqslant 13$. Our analysis shows that every CHM of order $4p$ with $p\equiv 1\bmod 4$ is equivalent to a Hadamard matrix with one of five distinct block structures, including Williamson type and (transposed) Ito matrices. If $p\equiv 3 \bmod 4$, then every CHM of order $4p$ is equivalent to a Williamson type or (transposed) Ito matrix.
\end{abstract}

\maketitle
\vspace*{-1cm}

\section{Introduction}
A Hadamard matrix (HM) of order $n$ is an  $n\times n$ matrix $H$ over $\{-1,1\}$ such that $HH^\intercal=nI_n$ where $I_n$ is the $n\times n$ identity matrix.
HMs have found numerous applications in areas such as cryptography, coding theory, and signal processing; we refer to the books of Horadam \cite{Horadam2} and Seberry~\cite{SeberryBook} for details and references.
Straightforward counting arguments show that the order of a HM is necessarily $1$, $2$ or $4t$ for $t \in \mathbb{N}$.
It is not known whether there exists a HM of order $4t$ for every value of $t$.
This is the \myem{Hadamard conjecture}, which is one of the best known problems in combinatorial design theory.

HMs are proven to exist for all $t \leq 166$, so at the time of writing the smallest order for which the existence of a HM seems to be unknown is $668$, see \cite{Horadam2}.
De Launey and Gordon \cite{deLauneyGordon} investigated the asymptotic number of HMs. They proved that $S(x)/x> (x/\log(x))\text{exp}((0.8178+o(1))(\log\log\log x)^2)$, where $S(x)$ is the number of positive integers $n\leq x$ such that there exists a HM of order $4n$; they also showed that this bound is hard to improve.
 Despite the challenges posed by the existence problem, the classification problem for HMs is well studied. While the existence of HMs for general order $4n$ is open, known classifications for small $n$ indicate that  the numbers of HMs of order $4n$ grow rapidly. It is therefore useful to study HMs up to equivalence, where matrices $A$ and $B$ are equivalent, written $A\equiv B$, if $A$ can be obtained from $B$ by row/column permutations and multiplication of rows/columns by $-1$.

The classification of HMs of orders less than 30, up to equivalence, was achieved through the efforts of numerous mathematicians in the 1980s and 1990s, see \cite{Spence}.
The classification of HMs of order 32 was completed in 2010 with the determination of $13,710,027$  equivalence classes, see \cite{KharaghaniEtAl2}.
The profusion of equivalence classes, even at small orders, motivates to look for  HMs of special types: as a generalisation of back-circulant matrices, a HM $H$ of order $4n$ is \myem{group-developed} if its entries are $h_{i,j}=\varphi(g_ig_j)$, where $G=\{g_1,\ldots,g_{4n}\}$ is a group of size $4n$ and $\varphi\colon G\to \{\pm 1\}$ is a map. It is known that the order of a group-developed HM is a perfect square, see also Lemma \ref{Wallis}, so there exist orders for which there are no group-developed HMs. Generalising group-developed HMs, \myem{cocyclic HMs} (CHMs) were introduced by de Launey and Horadam  \cite{deLaney-Horadam} as a class of HMs with additional algebraic properties; here the entries are defined by $h_{i,j}=\psi(g_i,g_j)\varphi(g_ig_j)$ where $\psi\colon G\times G\to \{\pm 1\}$ is a \emph{2-cocycle} and $\varphi\colon G\to \{\pm1\}$ a map; we give full details in Section \ref{seccoc}. The \myem{Cocyclic Hadamard Conjecture} \cite[Research Problem~38]{Horadam2} claims that for every $n$ there exists a CHM of order $4n$; in fact, many of the classical constructions, such as Ito and Williamson, are cocyclic. There are two excellent surveys of  cocyclic development in the literature:  de Launey and Flannery \cite{deLauneyFlannery} give an exposition of the theory of cocyclic development in its most general form; Horadam's monograph  \cite{Horadam2} contains an extensive discussion of cocyclic development for the special case of Hadamard matrices. \'{A}lvarez et al.\  \cite{SevAb, SevDih} have previously developed a computational framework for classifying Ito and Williamson CHMs.

\subsection{Main results}
We consider CHMs of order $4p$ with $p>3$ a prime. In Section \ref{secNot} we discuss preliminaries and introduce some notation. In Section \ref{secGroup} we recall results about the indexing group of a CHM of order $4p$. In Section \ref{secBLOCK} we independently recover the block structure proved in de Launey and Flannery \cite{deLauneyFlannery}: we prove in Theorem \ref{thmMain} that every CHM of order $4p$ is equivalent to a HM of the form
{\small\[
\HM(W,X,Y,Z)^{a,b}_{r,s,t} =\left[\begin{array}{rrrr}
W & X^a & Y^b & Z^{ab}\\
X & (-1)^rW^a & Z^b & (-1)^rY^{ab}\\
Y & (-1)^tZ^a & (-1)^sW^b & (-1)^{s+t}X^{ab}\\
Z & (-1)^{r+t}Y^a & (-1)^sX^b & (-1)^{r+s+t}W^{ab}
\end{array}\right]\]}for certain parameters $a,b,r,s,t$ and back-circulant $p\times p$ blocks $W,X,Y,Z$. We identify which structures yield equivalent HMs and show that, if $p\equiv 3 \bmod 4$, then all those CHMs are equivalent to a Williamson type or (transposed) Ito matrix, see Theorem \ref{propequivalences}; if $p\equiv 1\bmod 4$, then all those CHMs are equivalent to at least one of five block structures. In Section \ref{secClass} we  describe a new algorithm for classifying CHMs of order $4p$ up to equivalence; we implemented our algorithm in Magma and classified CHMs for primes $p\leq 13$.

\section{Preliminaries}\label{secNot}

\subsection{Central group extensions}\label{secExt}  All groups are finite and written multiplicatively. We denote by $C_n$ and $D_{2n}$ the cyclic group of size $n$ and dihedral group of size $2n$, respectively; we write $C_n^2=C_n\times C_n$. Making the usual identifications, we say a group $E$ is an \emph{extension} of a group $G$ by a group $N$ if the latter is a normal subgroup of $E$ with $E/N\cong G$. The extension is \emph{central} if $N$ lies in the centre $Z(E)$ of $E$. If $N$ is abelian, then every central extension of $G$ by $N$ is isomorphic to a group
\[E_\psi=\{(g,a)\colon g\in G, a\in N\}\quad\text{with multiplication}\quad (g,a)(h,b)=(gh,ab\psi(g,h)),\]
where $\psi\in Z^2(G,N)$ is a $2$-\emph{cocycle}, that is, a map $G\times G\to N$ with  $\psi(g,1)=\psi(1,g)=1$  and $\psi(g,hk)\psi(h,k)=\psi(g,h)\psi(gh,k)$ for all $g,h,k\in G$. Such a cocycle is a $2$-\emph{coboundary} if $\psi(g,h)=\tau(gh)\tau(g)^{-1}\tau(h)^{-1}$ for some map $\tau\colon G\to N$ with $\tau(1)=1$.
 The extension $E_{\psi}$ is \emph{split} if $E_{\psi}$ contains a subgroup isomorphic to $G$ which intersects trivially with $N$; this is denoted by $E_{\psi} = G \ltimes N$. It is known that  $E_{\psi}$ is split if and only if $\psi$ is a $2$-coboundary. For examples of cocycles we refer to \cite[Section 6.2.1]{Horadam2}.

\subsection{Williamson type and Ito matrices}\label{WILL-ITO}
An $n\times n$ matrix $M=(m_{r,c})$ is \emph{circulant} if $m_{r,c}=m_{0,c-r}$ for all $r,c=0,\dots,n-1$, where arithmetic is carried out modulo $n$. Similarly, $M = (m_{r,c})$ is \emph{back-circulant} if $m_{r,c} = m_{0,c+r}$ for all $r,c = 0, \dots, n-1$.  Let $R$ be a ring and let $G=\{g_0,\ldots,g_{n-1}\}$ be a group of order $n$ with fixed element ordering. The matrix $M$ is \emph{group-developed over $G$} (or $G$-developed) with coefficients in $R$ if there is a map $\phi\colon G\to R$ such that $M = (m_{r,c})$ with each $m_{r,c} = \phi(g_{r}g_{c})$. For example, a back-circulant matrix of order $n$ is $C_n$-developed, and a circulant matrix is equivalent to a group developed matrix. Group-developed HMs are precisely equivalent to \emph{Menon-Hadamard} difference sets, which have been studied extensively in the literature \cite[Section~2.3.1]{Horadam2}.

A \myem{Williamson type (Hadamard) matrix} (WTM) is a HM of the form \figa\ as in Figure 1, where the blocks $W, X, Y$, and $Z$ are $n\times n$ such that
\begin{equation}\label{wm1e1}
WW^\intercal+XX^\intercal+YY^\intercal+ZZ^\intercal=4nI_n
\end{equation}
and $AB^\intercal=BA^\intercal$ for all $A,B\in \{W,X,Y,Z\}$. If the blocks are symmetric circulant, then $H$ is a Williamson (Hadamard) matrix (WM), see the construction in  \cite[(8)]{Will}. Generalisations of WMs  with back-circulant, circulant but not symmetric, or group-developed blocks are known; a short historical overview of WTM is given in \cite{WMBAD}. Other variants and larger templates have been reported in \cite{GoethalsSeidal,SeberryBook}.
A particular generalisation of  WMs are  \myem{Ito (Hadamard) matrices} (IMs); these are HMs of the form \figb\ in Figure~1, where the blocks $W, X, Y, Z$ are $n\times n$  circulant, satisfying \eqref{wm1e1} and $WX^\intercal+YZ^\intercal = XW^\intercal+ZY^\intercal$.

\begin{figure}
  {\arraycolsep=1.5pt\def\arraystretch{0.8}
\begin{equation*}
{\rm\figa}:\quad\left(
\begin{array}{rrrr}
W & X & Y & Z\\
X & -W & Z & -Y\\
Y & -Z & -W & X\\
Z & Y & -X & -W
\end{array}
\right)\hspace*{2cm} {\rm\figb}:\quad\left(
\begin{array}{rrrr}
W & X & Y & Z\\
X & -W & Z & -Y\\
Y^\intercal & -Z^\intercal & -W^\intercal & X^\intercal\\
Z^\intercal & Y^\intercal & -X^\intercal & -W^\intercal
\end{array}
\right)
\end{equation*}

\vspace*{-2ex}

\caption{The block structure of Williamson and Ito matrices as defined in Section \ref{WILL-ITO}.}\label{figWM}
}
  \end{figure}

\subsection{Cocyclic  matrices}\label{cocyclicmatrices}\label{seccoc} CHMs are generalisations of group-developed HMs. Let $G=\{g_0,\ldots,g_{n-1}\}$ be a group and assume  $g_0=1$. A matrix $H$ over $\{\pm1\}$ of order $n$ is \emph{cocyclic} with \myem{indexing group}  $G$ if
\[H\equiv \left[\psi(g_i,g_j)\phi(g_ig_j))\right]_{i,j}\]
for some 2-cocycle $\psi\in Z^2(G,\LaRa)$ and map $\phi\colon G\to \LaRa$ with $\phi(1)=1$. Every HM of order at most 20 is cocyclic, see \cite[p. \ 86]{CatRod}.
For more examples of CHMs we refer to \cite[Section 6.4]{Horadam2}.
If $\phi$ is trivial, then $H$ is \myem{pure cocyclic}; if $\psi$ is trivial, then $H$ is group-developed. The following is a well-known observation.

\begin{lemma}\label{Cohomologous}\label{Wallis}
Let $H$ be a cocyclic matrix over $\{\pm 1\}$ with cocycle $\psi$ and map $\phi$; then the following hold:
\begin{ithm}
\item $H$ is equivalent to a pure cocyclic matrix with cocycle $\psi'(g,h)=\psi(g,h)\phi(gh)\phi(g)^{-1}\phi(h)^{-1}$.
\item If $H$ is group-developed and Hadamard, then the order of $H$ is a perfect square.
\end{ithm}
\end{lemma}
\begin{proof}
Note that
$[\psi'(g,h)]_{g,h\in G}$ differs from $[\psi(g,h)\phi(gh)]_{g,h\in G}$ only in the multiplication of rows and columns by scalars $\pm 1$; this proves a).
Every group-developed matrix contains the same number of elements $1$ in each row and column; let $s$ be the row  sum of $H$. If $J_n$ is the all-ones matrix of order, then $HJ_n = H^\intercal J_n= sJ_n$; but we also have $H^{\intercal}HJ_n = nJ_n$, from which it follows that $n = s^{2}$.
\end{proof}

\section{Cocyclic Hadamard matrices over a group of order $4p$}\label{secGroup}
Fix a (pure)  CHM $H$ of order $4p$, $p>3$ a prime, with indexing group $G$ and cocycle $\psi\in Z^2(G,\LaRa)$. In this case, $G$ is isomorphic to  $C_2^2\times C_p$ or $D_{4p}$, see  \cite[p.\ 223]{deLauneyFlannery}; details of the proof are required later, which is why we recall it here. The Sylow and Schur-Zassenhaus theorems show that  $G=K\ltimes N$ where $N\cong C_p$ and $K$ has order $4$. Let $E_\psi=G\ltimes_\psi \LaRa$ be the central extension defined by $\psi$. By Schur-Zassenhaus $E_{\psi}$ has a normal subgroup isomorphic to $N$ of the form $\hat{N}=\{(n,z_n)\mid n\in N\}$,  such that
\[E_\psi=\hat{K}\ltimes \hat{N}\] for $\hat{K}=\{(k,z)\mid k\in K,\; z\in \LaRa\}$; see \cite[Proposition 19.1.1]{deLauneyFlannery}. Note that $\hat{K}\cong K\ltimes_\psi \LaRa$, where, by abuse of notation, $\psi$ is identified with its restriction to $K\times K$. The next result is due to Ito.

\begin{theorem}[\cite{Ito}, Propositions~5~\&~6]\label{Ito}
The Sylow 2-subgroup $\hat{K}$ of $E_{\psi}$ is not cyclic or dihedral.
\end{theorem}
Thus   $\hat{K}$ has type $C_2^3$, $C_4\times C_2$, or $Q_8$. If $\hat{K}\cong C_2^3$, then $\hat{K}$ (and so also $E_\psi$) is a split extension of $K\cong C_2^2$ by $\LaRa$, hence $\psi$ is a coboundary. In this case, $H$ is necessarily group-developed, which is not possible since $|G|=4p$ is not a square, see Lemma \ref{Wallis}. This proves  the first claim of the following lemma.

\begin{lemma}\label{classification}
The Sylow 2-subgroup $\hat{K}$ of $E_{\psi}$ is $Q_8$ or $C_4\times C_2$, and it acts on $\hat{N}$ trivially or by inversion. The indexing group of a CHM of order $4p$ with $p>3$ a prime is isomorphic to $C_2^2\times C_p$ or $D_{4p}$.
\end{lemma}
\begin{proof}
The  central extensions of $K=C_4$ by $\LaRa$, namely $C_{4} \times C_{2}$ and $C_{8}$, are discarded by Theorem \ref{Ito} and Lemma \ref{Wallis}. The central extensions of $K=C_2^2$ by $\LaRa$ can be written as $\hat{K}=K\ltimes_\psi \LaRa=L_{r,s,t}$ where the generators of $L_{r,s,t}=\langle a,b,z\rangle$ satisfy
\begin{equation*}\label{DefLIJK}
a^2=z^r,\; b^2=z^s,\; b^a=bz^t,\; z^a=z^b=z;
\end{equation*}
a direct computation shows that $L_{0,0,1}, L_{0,1,1}, L_{1,0,1}\cong D_8$ and $L_{1,0,0},L_{0,1,0},L_{1,1,0}\cong C_4\times C_2$ and $L_{0,0,0}\cong C^3_2$, and $L_{1,1,1}\cong Q_8$. By Theorem \ref{Ito} and Lemma \ref{Wallis},  the possibilities for $\hat{K}$ are  $L_{1,1,1}$, $L_{1,0,0}$, $L_{0,1,0}$, $L_{1,1,0}$; in particular, $K\cong C_2^2$ and so the elements of $K$ act trivially or by inversion on $N$.
\end{proof}

\begin{remark}\label{remcoc}Mapping a 2-cocycle in $Z^2(K,\LaRa)$ to the triple $(r,s,t)\in \mathbb{Z}_2^3$, as described in the proof of Lemma \ref{classification}, is a group homomorphism with nontrivial kernel $\langle \kappa\rangle\cong C_2$. Since $|Z^2(K,\LaRa)|=16$, this shows that for each triple $(r,s,t)$ there exist exactly two 2-cocycles in $Z^2(K,\LaRa)$ that map onto $(r,s,t)$. For  each $(r,s,t)\in\{(1,1,1),(1,0,0),(0,1,0),(1,1,0)\}$ we fix a 2-cocycle $\psi_{r,s,t}\in Z^2(G,\LaRa)$ such that $E_{\psi_{r,s,t}}$ decomposes as $\hat{K}\ltimes \hat{N}$ with $\hat{K}=K\ltimes_{\psi_{r,s,t}}\LaRa=L_{r,s,t}$ as in the proof of Lemma \ref{classification}; specifically, if we order the elements of $K$ by $1,a,b,ab$, then we can assume that
{\arraycolsep=1.5pt\def\arraystretch{0.8}\[[\psi_{r,s,t}(u,v)]_{u,v\in K}=\left[\begin{array}{rrrr}
1 & 1 & 1 & 1\\
1 & (-1)^r & 1 & (-1)^r\\
1 & (-1)^t & (-1)^s & (-1)^{s+t}\\
1 & (-1)^{r+t} & (-1)^s & (-1)^{r+s+t}
  \end{array}\right]
\quad\text{and}\quad [\kappa(u,v)]_{u,v\in K}=\left[\begin{array}{rrrr}
1 & 1 & 1 & 1\\
1 & 1 & -1 & -1\\
1 & -1 & 1 & -1\\
1 & -1 & -1 & 1
  \end{array}\right].\]
Now let $\psi\in Z^2(G,\LaRa)$ such that the restriction to $K\times K$ gives parameters $(r,s,t)$, that is, $\psi|_{K\times K}$ is either $\rho$ or $\rho\cdot\kappa$ where $\rho=\psi_{r,s,t}|_{K\times K}$. We claim that $\delta=\psi\psi_{r,s,t}^{-1}$ lies in $B^2(G,\LaRa)$. For this recall from above that $E_\delta$ can be written as $E_\delta=\hat K\ltimes\hat N$ with $\hat{N}\cong C_p$ and $\hat K\cong K\ltimes_{\delta} \LaRa$. Since $\delta|_{K\times K}$ lies in $B^2(K,\LaRa)$, we know that $\hat K$ splits, hence $E_\delta$ splits; this proves that $\delta\in B^2(G,\LaRa)$, as claimed.}
  \end{remark}

\section{Block structure}\label{secBLOCK}
We now study the block structure of CHMs of order $4p$ with $p>3$ a prime, and prove our main results Theorems \ref{thmMain} and \ref{propequivalences}.
For a CHM of order $4p$ there are sixteen possible  block structures, corresponding to the groups $L_{1,0,0}$, $L_{0,1,0}$, $L_{1,1,0}$ and $L_{1,1,1}$,
in combination with four possible actions by the Sylow 2-subgroup of the indexing group on the blocks of the array;
IMs and WMs are included in that list.

In the following let $H$ be a matrix over $\{\pm1\}$ of order $n$. Let $G$ be a group of size $n$ and fix an ordering of the elements of $G$ to index the rows and columns of $H$. The Kronecker product $A\otimes B$ of two matrices is the matrix where each entry $a_{i,j}$ of $A$ is replaced by the matrix $a_{i,j}B$. The \emph{expanded matrix} of $H$ is\[\ext_H=\left(\begin{array}{rr}1 & -1\\ -1 & 1 \end{array}\right)\otimes H.\]

\begin{theorem}[\cite{CatRod}]\label{Equivalence1}
With the previous notation, $H$ is cocyclic with $\psi\in Z^2(G,\LaRa)$  and indexing group $G$ if and only if $\ext_H$ is $E_\psi$-developed.
\end{theorem}

If $H$ is a CHM with indexing group $G$, say $H= [\psi(g,h)\tau(gh)]_{g,h\in G}$, then $H\equiv H'= [\psi'(g,h)]_{g,h\in G}$, where ${\psi'}=\psi\cdot \delta$ for the $2$-coboundary $\delta(g,h)=\tau(gh)\tau(g)^{-1}\tau(h)^{-1}$, see Lemma \ref{Wallis}. The matrices $\ext_{H}$ and $\ext_{H'}$ differ only in the ordering of rows and columns. Since $E_\psi\cong E_{{\psi'}}$, we have that $\ext_H$ is $E_\psi$-developed if and only if is $E_{\psi'}$-developed.

We recall the following facts from the proof of Theorem \ref{Equivalence1} given in \cite{CatRod}. First, if $H= \left[\psi(g,h)\right]_{g,h\in G}$ is pure cocyclic, then we can assume that $\ext_H=\left[\phi((gh,ab\psi(g,h)))\right]_{(g,a),(h,b)\in E_\psi}$ where $\phi\colon E_\psi\to \LaRa$ is given by $\phi(g,a)\mapsto a$. Second, if $\ext_H$ is $E_\psi$-developed, say with map $\phi\colon E_\psi\to\LaRa$, then
\begin{eqnarray}\label{ExtMatrix2}H = \left[\psi(g,h)\phi'(gh) \right]_{g,h\in G}
\end{eqnarray}
for some function $\phi'(g)=\phi(g,1)$ for $g\in G$. Up to Hadamard equivalence, Lemma \ref{Wallis} shows that $H \equiv \left[ \psi'(g,h)\right]$ where $\psi'$ is cohomologous to $\psi$.

\begin{remark}
We illustrate the general theory: let $H$ be a CHM of order $4p$ with indexing group $C_{2}^{2} \times C_{p}$ and cocycle corresponding to  $E_\psi=Q_8\times C_p$. Theorem~\ref{Equivalence1} shows that $\mathcal{E}_{H} = \left(\mu(gh) \right)_{g,h \in E_\psi}$ for some $\mu\colon E_\psi  \rightarrow \{\pm 1\}$.
Fix a total order on $E_\psi$ by using the ordering $1, i, j, k, -1, -i, -j, -k$ on $Q_{8}$, and setting
$(\alpha, n_{1}) < (\beta, n_{2})$ if $\alpha < \beta$ in the ordering on $Q_{8}$ or $\alpha = \beta$ and $n_{1} < n_{2}$ (as integers).
Up to permutation equivalence, the rows and columns of $\mathcal{E}_{H}$ are labelled by the elements of $E_\psi$ in this order. Now the cosets of $C_{p}$ in $E_\psi$ label back-circulant block matrices, and these block matrices are ordered according to the Cayley table of $Q_{8}$. Furthermore, the proof of Theorem~\ref{Equivalence1} shows that $\mu( -\alpha, n) = -\mu(\alpha, n)$ for all $(\alpha, n) \in E_\psi$; thus the $4n \times 4n$ matrix in the upper left of
$\mathcal{E}_{H}$ has the form \figa\ in Figure \ref{figWM}.
\end{remark}

\subsection{Block structure}\label{secBlock}
We now prove our main theorem.

\begin{theorem}\label{thmMain}
Every cocyclic Hadamard matrix of order $4p$ with $p>3$ a prime is Hadamard equivalent to a Hadamard matrix of the form
\begin{equation}\label{blockstructure}
\HM(W,X,Y,Z)^{a,b}_{r,s,t} = \left[\begin{array}{rrrr}
W & X^a & Y^b & Z^{ab}\\
X & (-1)^rW^a & Z^b & (-1)^rY^{ab}\\
Y & (-1)^tZ^a & (-1)^sW^b & (-1)^{s+t}X^{ab}\\
Z & (-1)^{r+t}Y^a & (-1)^sX^b & (-1)^{r+s+t}W^{ab}
\end{array}\right]
\end{equation}
where $(r,s,t)\in \{(1,0,0),(0,1,0),(1,1,0),(1,1,1)\}$, each $M\in\{W,X,Y,Z\}$ is a back-circulant $p\times p$ block, and $a,b\in\{\Tri,\Inv\}$ with $\Inv\Inv=\Tri=\Tri\Tri$ and $\Inv=\Inv\Tri=\Tri\Inv$ such that $M^\Tri=M$ and $M^\Inv$ is a circulant matrix with the same first row as $M$.
\end{theorem}
\begin{proof}
  Let $H$ be a CHM with cocycle $\psi\in Z^2(G,\LaRa)$, where $G=K\ltimes N$ with $K\cong C_2^2$ and $N\cong C_p$. It follows from Theorem \ref{Equivalence1} that $\ext_H$ is $E_\psi$-developed. Looking at the restriction of $\psi$ to $K\times K$, it follows from Remark \ref{remcoc} that there exist  $(r,s,t)\in\{(1,1,1),(1,1,0),(1,0,0),(0,1,0)\}$ and  $\delta\in B^2(G,\LaRa)$ such that $\psi=\psi_{r,s,t}\cdot \delta$  defines a group extension equivalent to $E_{\psi_{r,s,t}}$; moreover, the restriction $\delta_{K\times K}$ is either trivial or equals $\kappa\in B^2(K,\LaRa)$ as defined in Remark \ref{remcoc}. Note that, by definition of $\psi_{r,s,t}$, the sign  pattern visible in \eqref{blockstructure} equals $[\psi_{r,s,t}(u,v)]_{u,v}\in K$, where the elements of $K$ are ordered as $1,a,b,ab$.

    As above, write $E_{\psi}= \hat{K}\ltimes \hat{N}$ where $\hat{N}=\{(n,z_n)\mid n\in N\}\cong N$ and $\hat{K}= \{(k,z)\mid k\in K,\; z\in \LaRa\}$. Let $\phi$ be the map of the  $E_{\psi}$-development of $\ext_H$. Since $T=\{(kn,z_{n})\mid k\in K,\;  n\in N\}$ is a  transversal of $\LaRa$ in $E_{\psi}$, we have $H\equiv [\phi(u v)]_{u,v\in T}$. The discussion around \eqref{ExtMatrix2} in Theorem \ref{Equivalence1} now proves that
\small    \arraycolsep=2pt\def\arraystretch{0.4}
  \begin{eqnarray*}
H&\equiv&    \left[\psi(k,h)\phi\left( khn^{h}m, z_{n^hm}\right)\right]_{(kn,z_n),(hm,z_m)\in T.}
\end{eqnarray*}
For $k\in K=\{1,a,b,ab\}$ and $n\in N$ we define $\phi'(kn)=\phi(kn,z_{n})$, where $(kn,z_{n})\in E_{\psi}$. Ordering the elements of $T$ as cosets of $\hat{N}$ allows us to exhibit a $4\times 4$-block structure
{\small    \arraycolsep=2pt\def\arraystretch{0.4}
  \begin{eqnarray}\label{strucmat}
H&\equiv&M_{\phi',\psi}=\left[\begin{matrix}
\left[\phi'\left( nm\right)\right] &
\left[\phi'\left( an^{a}m\right)\right] &
\left[\phi'\left( bn^{b}m\right)\right] &
\left[\phi'\left( abn^{ab}m\right)\right]\\ \\
\left[\phi'\left( anm\right)\right] &
\left[\psi(a,a)\phi'\left( n^{a}m\right)\right]&
\left[\psi(a,b)\phi'\left( abn^{b}m\right)\right]&
\left[\psi(a,ab)\phi'\left( bn^{ab}m\right)\right]\\ \\
\left[\phi'\left( b,nm\right)\right] &
\left[\psi(b,a)\phi'\left( abn^{a}m\right)\right]&
\left[\psi(b,b)\phi'\left( n^{b}m\right)\right]&
\left[\psi(b,ab)\phi'\left( an^{ab}m\right)\right]\\ \\
\left[\phi'\left( abnm\right)\right] &
\left[\psi(ab,a)\phi'\left( bn^{a}m\right)\right]&
\left[\psi(ab,b)\phi'\left( an^{b}m\right)\right]&
\left[\psi(ab,ab)\phi'\left( 1n^{ab}m\right)\right]\\
\end{matrix}\right]
\end{eqnarray}}where each of the blocks is indexed by $n,m\in N$.

We first look at the signs of those blocks, defined by $\psi(u,v)$ with $u,v\in K$. As mentioned above, the restriction of   $\psi=\psi_{r,s,t}\cdot\delta$ to $K\times K$  is uniquely determined by $(r,s,t)$, modulo $\kappa$. By Remark \ref{remcoc}, if $\delta_{K\times K}$ is trivial, then those signs are exactly those visible in \eqref{blockstructure}; if $\delta_{K\times K}=\kappa$, then $M_{\phi',\psi}\equiv M_{\phi'',\psi_{r,s,t}}$ where the function $\phi''$  equals $\phi'$ except on the coset defined by $ab$, where it takes the opposite value to $\phi'$. (This is simply a relabelling of blocks, the matrix entries are unchanged; the final equivalence follows by negating the last $p$ rows and $p$ columns.) Thus, up to equivalence, we can assume that $\delta|_{K\times K}=1$ is trivial, hence the signs determined by $\psi$ equal those in \eqref{blockstructure}.

  Now we look at the structure of the blocks.  Since $N$ is cyclic and $v \in K$ acts trivially or by inversion on $N$, it follows that each block $[\phi'(un^vm)]_{n,m\in N}$ is circulant or back-circulant for any $u \in K$. Thus, defining $W=\left[\phi'\left( nm\right)\right]$, $X=\left[\phi'\left( an^{a}m\right)\right]$, $Y=\left[\phi'\left( bn^{b}m\right)\right]$, and $Z=\left[\phi'\left( abn^{ab}m\right)\right]$, we have proved that $H\equiv \HM(W,X,Y,Z)^{a,b}_{r,s,t}$ is as defined in the  theorem.
\end{proof}

\subsection{Equivalence}\label{secEquivalence}We  discuss when matrices of the form \eqref{blockstructure} are equivalent, and prove Theorem \ref{propequivalences}.

\begin{proposition}\label{propequiv}
Let $H=\HM(W,X,Y,Z)^{a,b}_{r,s,t}$ be a Hadamard matrix.
\begin{ithm}

\item If $\alpha,\beta,\gamma,\delta\in\{\pm 1\}$ with $\alpha\beta\gamma\delta=1$, then  $H$ is equivalent to each of
\begin{eqnarray}
\label{eq1}&&\HM(Y,Z, W,X)^{a,b}_{r,s,t},\;  \HM(X,W,Z,Y)^{a,b}_{r,s,t},\; \HM(Z,Y,X,W)^{a,b}_{r,s,t},\; \HM(\alpha W,\beta X, \gamma Y ,\delta Z)^{a,b}_{r,s,t}.
\end{eqnarray}
\item Let $(r,s,t)=(1,1,1)$. If  $a,b\in\{\Tri,\Inv\}$, then
\begin{align*}
  &\HM(W,X,Y,Z)^{{ a},{ b}}_{1,1,1}\equiv\HM(W,Y,Z,X)^{{ b},{ ab}}_{1,1,1}\equiv\HM(W,Z,X,Y)^{{ ab},{ a}}_{1,1,1}\\
   &\HM(W,X,Y,Z)^{{ a},{ b}}_{1,1,1}\equiv\HM(-W,Y,X,Z)^{{ b},{ a}}_{1,1,1}\equiv\HM(-W,Z,Y,X)^{{ ab},{ b}}_{1,1,1}.
\end{align*}
If $(a,b)=(\Tri,\Tri)$, then $H$ is equivalent to a WM; otherwise $M^\intercal$ is equivalent to an IM.
\item If $(r,s,t)\in \{(1,0,0),(0,1,0),(1,1,0)\}$, then  $p\equiv 1 \bmod 4$ and for all $a,b\in\{\tau,\iota\}$
\[ \HM(W,X,Y,Z)^{{ a},{ b}}_{0,1,0}\equiv\HM(W,Y,X,Z)^{{ b},{ a}}_{1,0,0}\equiv\HM(W,-Y,-Z,-X)^{{ b},{ ab}}_{1,1,0}.\]
\item If $(r,s,t)=(1,1,0)$, then  $p\equiv 1 \bmod 4$  and $\HM(W,X,Y,Z)^{{ \Tri},{ \Inv}}_{1,1,0}\equiv\HM(W,Y,X,Z)^{{ \Inv},{ \Tri}}_{1,1,0}$.
\end{ithm}
\end{proposition}

\begin{proof}
Denote by $\diag(A,B,C,D)$ the block diagonal matrix with diagonal  $A,B,C,D$.  Let $P$ be the $p\times p$ matrix with 1s on its back-diagonal and 0s elsewhere, and let $Q$ be the $p\times p$ permutation matrix (acting on rows) of the $p$-cycle $(1,2,\ldots,p)$. Note that if $U$ is a back-circulant $p\times p$ matrix with first row $u$, then $PUQ^\intercal$ is a circulant matrix with first row $u$; moreover, if $V$ is a circulant $p\times p$ matrix, then $V=QV Q^\intercal$. All permutation matrices defined below are $4\times 4$ and act on rows.

\begin{iprf}

\item Recall that $H=\HM(W,X,Y,Z)^{a,b}_{r,s,t}$ has parameters $a,b\in\{\Tri,\Inv\}$ and $r,s,t\in\{0,1\}$. We define a few auxiliary matrices. Let $T_1',T_2',T_3'$ be the permutation matrices of $(1,3)(2,4)$, $(1,2)(3,4)$, and $(1,4)(2,3)$, respectively.  Define $4\times 4$ matrices $S_1',S_2',S_3'$ as \[\diag(1,(-1)^t, (-1)^s,(-1)^{s+t}),\; \diag(1,(-1)^r, 1,(-1)^r), \text{ and }\diag(1,(-1)^{r+t}, (-1)^s,(-1)^{r+s+t}),\] respectively.  Now define $T_d= T_d^\prime\otimes I_4$ and $S_d= S_d^\prime\otimes I_4$ for $d=1,2,3$, so that $T_dHS_d^\intercal$ is
  \[\HM(Y,Z, W,X)^{a,b}_{r,s,t},\quad \HM(X,W,Z,Y)^{a,b}_{r,s,t},\quad \HM(Z,Y,X,W)^{a,b}_{r,s,t},\]
  respectively. Selective multiplication of rows and columns by $-1$ yields the last part of \eqref{eq1}; e.g.\ $\alpha=\beta=-1$ and $\delta=\gamma=1$ corresponds to multiplying columns $1$ to $2n$ and  rows $2n+1$ to $4n$  of $H$ by $-1$.

\item Let  $T= T^\prime\otimes I_4$ where $T^\prime$ is  the permutation matrix of $(2,3,4)$. Now  $TMT^\intercal =\HM(W,Y,Z,X)^{{ b},{ ab}}_{1,1,1}$ and $T\HM(W,Y,Z,X)^{{ b},{ ab}}_{1,1,1}T^\intercal = \HM(W,Z,X,Y)^{{ ab},{ a}}_{1,1,1}$. If $(a,b)=(\Tri,\Tri)$, then define  $L=\diag(P,P,P,P)$ and $R=\diag(Q,Q,Q,Q)$; then $LHR^\intercal$ has shape \figa. For  $(a,b)=(\Inv,\Inv)$ let $S'$ be  the  permutation matrix of $(2,4)$, and define  $L=\diag(P,Q,Q,P)$. For $(a,b)=(\Inv,\Tri)$ let $S'$  be  the permutation matrix of $(3,4)$, and let  $L=\diag(P,P,Q,Q)$. For  $(a,b)=(\Tri,\Inv)$ let $S'$ be  the permutation matrix of $(2,3)$, and define $L= \diag(P,Q,P,Q)$. In all cases, $R=I_4 \otimes Q$ and $S=S'\otimes I_p$ bring $SLH^\intercal R^\intercal S^\intercal$ in the form \figb.
Let $T_1'$ and $T_2'$ be the permutation matrices of $(2,3)$ and $(2,4)$, respectively, and define $T_d=T_d'\otimes I_p$. It follows that
$T_1 \HM(W,X,Y,Z)^{{ a},{ b}}_{1,1,1}T_1^\intercal =\HM(W,Y,X,-Z)^{{ b},{ a}}_{1,1,1}$ and $T_2\HM(W,X,Y,Z)^{{ a},{ b}}_{1,1,1}T_2^\intercal=\HM(W,Z,Y,-X)^{{ ab},{ b}}_{1,1,1}$.

\item  A counting argument \cite[Section 19.2.2]{deLauneyFlannery} shows that $p$ must be the sum of two squares, and the latter holds if and only if $p\equiv 1 \bmod 4$ by Fermat's theorem on sums of two squares.
If $S'$ is the permutation matrix of $(2,3)$, then $S=S'\otimes I_p$ satisfies $S\HM(W,X,Y,Z)^{a,b}_{0,1,0} S^\intercal = \HM(W,Y,X,Z)^{b,a}_{1,0,0}$. Now let  $S'$ be the permutation matrix of $(3,4)$ and multiply its top-left entry by $-1$.  Then $S=(-S')\otimes I_p$ satisfies  $S\HM(W,X,Y,Z)^{ a, b}_{1,1,0} S^\intercal=\HM(W,-X,-Z,-Y)^{a,ab}_{1,0,0}$.

\item Let $T^\prime$ be  the permutation matrix corresponding to the cyclic permutation $(2,3)$, and define $T= T^\prime\otimes I_4$. It follows that  $T\HM(W,X,Y,Z)^{{ \Tri},{\Inv}}_{1,1,0}T^\intercal =\HM(W,Y,X,Z)^{{ \Inv},{ \Tri}}_{1,1,0}$.\qedhere
\end{iprf}
\end{proof}

Let $H=\HM(W,X,Y,Z)^{a,b}_{r,s,t}$  and denote by $w,x,y,z$ be the sums of the first rows of the blocks $W,X,Y,Z$, respectively.  Let $J$ be the all-ones matrix of order $p$; multiplying  \eqref{wm1e1} by $J$ yields
\begin{equation}\label{sum}
w^2+x^2+y^2+z^2=4p.
\end{equation}
A quadruple $(w,x,y,z)$ of integers with \eqref{sum} is a \emph{decomposition} of $4p$.

\begin{theorem}\label{propequivalences}
Let $H$ be  CHM of order $4p$ with $p>3$ a prime.
\begin{ithm}
\item The matrix $H$ is equivalent to some $\HM(W,X,Y,Z)^{a,b}_{r,s,t}$ where
\begin{ithm}
\item[$\bullet$] $(a,b)\in\{(\Tri,\Tri),(\Tri,\Inv)\}$ and $(r,s,t)=(1,1,1)$,
\item[$\bullet$] or $p\equiv 1\bmod 4$ and  $(a,b)\in\{(\Tri,\Tri),(\Tri,\Inv),(\Inv,\Inv)\}$ and $(r,s,t)=(1,1,0)$.
\end{ithm}
\item[\rm{b})]  If $p\equiv 3\bmod 4$, or if $p\equiv 1\bmod 4$ and the block row sums $w,x,y,z$ satisfy $wz\ne xy$, then $H$ is equivalent to a WTM or transposed IM of the form $\HM(W,X,Y,Z)^{a,b}_{1,1,1}$.
\item[\rm{c)}] CHMs of the form  $\HM(W,X,Y,Z)^{a,b}_{1,1,0}$ can only exist if $p\equiv 1\bmod 4$ and the row sums satisfy $(w,x,y,z)=(w,\alpha w,z,\alpha z)$ or $(w,x,y,z)=(w,\alpha z, w,\alpha z)$, where $2p=w^2+z^2$ and $\alpha\in \{\pm 1\}$.
\end{ithm}
In all cases, the matrices $W, X, Y, Z$ are $p\times p$ back-circulant.
\end{theorem}

\begin{proof}
\begin{iprf}
\item Theorem \ref{thmMain} gives the general structure of a CHM, so the first claim is  Proposition \ref{propequiv}.
\item If $(r,s,t) \neq (1,1,1)$ then $p\equiv 1\bmod 4$ and up to equivalence $H=\HM(W,X,Y,Z)^{a,b}_{1,1,0}$, by part (c) of Proposition \ref{propequiv}.
Then $HH^\intercal=4pI_{4p}$ yields  $W(Z^{ab})^\intercal+Z(W^{ab})^\intercal-X(Y^{ab})^\intercal-Y(X^{ab})^\intercal=0$; multiplication by the all-ones matrix gives the condition $wz=xy$ on the row sums of $W, X, Y, Z$. So, if $wz\ne xy$, then $(r,s,t)=(1,1,1)$, and $H$ is equivalent to a WTM or transposed~IM.
\item By part b), we have $wz=xy$.  Substituting  $z=xy/w$ into $w^2+x^2+y^2+z^2=4p$ and solving for $x^2$  yields $x^2=\frac{4pw^2}{w^2+y^2} - w^2$. Since this is an integer, $w^2+y^2$ divides $4pw^2$; we proceed with a case distinction.
 If $p$ divides $w^2+y^2$, then $w^2+x^2+y^2+z^2=4p$ shows that $p$ also divides $x^2+z^2$. If $p$ does not divide  $w^2+y^2$, then $w^2+y^2$ must divide $4w^2$. The divisors of $4w^2$ greater than $w^2$ are $2w^2$ and $4w^2$, so we have $y^2=w^2$ or $y^2=3w^2$. In the former case, $w^2+y^2 = 2w^2$, so $x^2$ must be $2p-w^2$; in particular, $p$ divides $x^2+w^2$. In the latter case, $x^2$ must be $p-w^2$, so $p$ divides $x^2+w^2$. Either way, $p$ divides $x^2+w^2$, and then also $z^2+y^2$. As a result, we must have $w^2+x^2 \in \{p,2p,3p\}$ and $y^2+z^2 \in \{3p,2p,p\}$, or   $w^2+y^2 \in \{p,2p,3p\}$ and $x^2+z^2 \in \{3p,2p,p\}$. By the Two-Square Theorem \cite[p.\ 140-142]{beiler}, we cannot write $3p$  as the sum of two squares, which forces  $w^2+x^2 = 2p = y^2+z^2$   or   $w^2+y^2 = 2p = x^2+z^2$. As shown in \cite[p.\ 141]{beiler}, up to signs and order, $2p$ can be written uniquely as a sum of two squares. The equation $wz=xy$ together with a case distinction proves the claim. \qedhere
\end{iprf}
\end{proof}

\section{Classification}\label{secClass}
We describe a new algorithmic approach to classify CHMs of order $4p$.  By Theorem \ref{propequivalences}, every such CHM is equivalent to a matrix $H=\HM(W,X,Y,Z)^{a,b}_{r,s,t}$ with  $(a,b)\in\{(\Tri,\Tri),(\Tri,\Inv))\}$ and $(r,s,t)=(1,1,1)$, or $p\equiv 1\bmod 4$ and $(a,b)\in\{(\Tri,\Tri),(\Tri,\Inv),(\Inv,\Inv)\}$ and  $(r,s,t)=(1,1,0)$.  We now describe two methods that can be used to trim the search spaces. As before, denote by $w,x,y,z$ the corresponding block row sums. Since $p>3$, we have  $0<w,x,y,z < p$.

\subsection{Eigenvalues}\label{secM1} The first method is based on  \cite{Holzamann-etal}. The \emph{Gramian} of a square matrix $U$ with entries $\pm1$ is  $UU^\intercal$. If $U$ and $V$ are circulant and back-circulant, respectively, with the same first row of length $p$, then  $C=UU^\intercal=VV^\intercal$ is symmetric and circulant, with first row of the form $(c_0,c_1,c_2,\ldots,c_2,c_1)$. If $U$ is a block of $H$, then \eqref{sum} implies that the first row of $U$ is not all $1$s or all $-1$s; now \cite[Proposition~1.1]{Ingleton} proves that $U$ has rank $p$, so its Gramian  is positive definite. Together with $HH^\intercal=4pI_{4p}$, we deduce that
\begin{equation}\label{gramequation}
WW^\intercal+XX^\intercal+YY^\intercal+ZZ^\intercal = 4pI_p,
\end{equation}
independent of the actions of $a$ and $b$. Let $\omega$ be a primitive $p$-th root of unity and let  $T$ be the permutation matrix of the $p$-cycle $(1,2,\ldots,p)$. For each $q\in\{0,\ldots,p-1\}$ the vector $v_{q}=((\omega^q)^0,(\omega^q)^1,\dots, (\omega^q)^{p-1})$ is an eigenvector of $T$ with eigenvalue $\omega^{p-q}$. Let $U\in\{W,X,Y,Z\}$. By construction, $UU^\intercal$ is  symmetric and circulant, hence a polynomial in $T$. It follows from the above discussion that $UU^\intercal$ has eigenvectors $\{v_0,\ldots,v_{p-1}\}$, and the eigenvalue corresponding to $v_q$ is
\begin{equation*}
\lambda_{q,U} = (u_0\cdot u_0) +\sum\nolimits_{t=1}^{(p-1)/2} 2(u_0\cdot u_t) \cos\left(2qt\pi/p\right)>0,
\end{equation*}
where $u_0\cdot u_t$ is the dot product of the rows $u_0$ and $u_t$ of $U$; note that these dot products are the entries in the first row of $UU^\intercal$. If  $B\subseteq \{W,X,Y,Z\}$, then \eqref{gramequation} shows
\begin{equation}\label{eigensum}
\sum\nolimits_{U\in B}\lambda_{q,U} \leqslant 4p
\end{equation}
for all $q=0,\dots,p-1$; these inequalities help to trim the search space significantly.

\subsection{Linear equations}\label{secM2}For the second method, let $U$ and $V$ be back-circulant with rows $u_0,\ldots,u_{p-1}$ and $v_0,\ldots,v_{p-1}$, respectively. Recall that $a,b\in\{\Tri,\Inv\}$, and denote by $U^\Inv$ and $V^\Inv$ the circulant matrices which share the same first row with $U$ and $V$, respectively. As above, let  $T$ be the permutation matrix of the $p$-cycle $(1,2,\ldots,p)$.  A straightforward calculation shows that $UV^\intercal = V^\Inv(U^\Inv)^\intercal= \sum\nolimits_{l=0}^{p-1}(u_0\cdot v_l)T^l$ and $VU^\intercal = \sum\nolimits_{l=0}^{p-1}(u_0\cdot v_{p-l})T^l$ are both circulant. Recall that  $H=\HM(W,X,Y,Z)_{r,s,t}^{a,b}$ and suppose  the blocks  $W,X,Y,Z$ have rows $w_i,x_i,y_i,z_i$ with $i=0,\ldots,p-1$, respectively. If $H$ is a HM, then, similar to \eqref{gramequation}, the condition $HH^\intercal=4pI_{4p}$ yields other equations involving these blocks, such as
\[WX^\intercal +(-1)^r X^a(W^a)^\intercal+Y^b(Z^b)^\intercal+(-1)^rZ^{ab}(Y^{ab})^\intercal = 0I_p.\]
Via the above formulas, this translates to $\sum\nolimits_{l=0}^{p-1}\left(w_0x_l+(-1)^rw_0x_{l^{\hat{a}}}+z_0y_{l^{\hat{b}}}+(-1)^rz_0y_{l^{ab}}\right)T^l = 0 I_p,$
where for $l\in\{0,\ldots,p-1\}$ we define \[l^\Tri=l^{\widehat{\Inv}}=l\quad\text{and}\quad l^\Inv=l^{\widehat{\Tri}}=p-l.\]
It follows that for all $l$ we have $w_0x_l+(-1)^rw_0x_{l^{\widehat{a}}}+z_0y_{l^{\widehat{b}}}+(-1)^rz_0y_{l^{ab}}=0$,
that is
\begin{eqnarray}
\label{eq4}&&z_0\left(y_{l^{\widehat{b}}}+(-1)^ry_{l^{ab}}\right)=- w_0\left(x_l-(-1)^rx_{l^{\widehat{a}}}\right).
\end{eqnarray}
Similar equations follow from considering other $p\times p$ blocks in the product $HH^\intercal=4p I_{4p}$, namely
\begin{equation}\label{eq5}
\begin{aligned}
z_0\left(y_{l^{\widehat{b}}}+(-1)^ry_{l^{ab}}\right)&=- w_0\left(x_l+(-1)^rx_{l^{\widehat{a}}}\right)\\
z_0\left((-1)^t x_{l^{\widehat{a}}}+(-1)^{s+t}x_{l^{ab}}\right) &=-w_0\left(y_{l}+(-1)^s y_{l^{\widehat{b}}}\right)\\
z_0\left(w_{p-l}+(-1)^{r+s+t}w_{l^{ab}}\right) &=-x_0\left((-1)^{r+t}y_{l^a}+(-1)^sy_{l^{\widehat{b}}}\right)\\
z_0\left((-1)^{r+t}w_{l^{\widehat{a}}}+(-1)^sw_{l^b}\right) &=-x_0\left(y_{l}+(-1)^{r+s+t}y_{l^{\widehat{ab}}}\right)\\
z_0\left(x_{p-l}+(-1)^s  x_{l^b}\right) &=-w_0\left((-1)^t y_{l^a}+(-1)^{s+t}y_{l^{\widehat{ab}}}\right)\\
z_0\left(y_{p-l}+(-1)^ry_{l^a}\right)&=- w_0\left(x_{l^b}+(-1)^rx_{l^{\widehat{ab}}}\right)\\
z_0\cdot (1,\ldots,1)&= z.
\end{aligned}
\end{equation}
Given $W,X,Y$ and fixing $r,s,t,a,b$, Equations \eqref{eq4} -- \eqref{eq5} yield a linear equation system for $z_0$, whose solution space contains all first rows $z_0$ of $Z$ such that $\HM(W,X,Y,Z)^{a,b}_{r,s,t}$ is a CHM.

\subsection{Algorithm} Based on Proposition \ref{propequiv}, Theorem \ref{propequivalences}, and the ideas described in Sections \ref{secM1} and \ref{secM2},  Algorithm \ref{alg1} describes a method to classify all HMs $\HM(W,X,Y,Z)^{a,b}_{r,s,t}$ of order $4p$ with  $p>3$ a prime up to equivalence; as above, $w,x,y,z$ are the corresponding block row sums.

{\small
\begin{algorithm}
\caption{An algorithm to classify all CHMs of order $4p$ up to equivalence}\label{alg1}
\begin{algorithmic}[1]
  \Require a prime $p>3$
\Ensure a list of all CHMs of order $4p$, up to equivalence
 \State initialise $L$ as the empty list
\State determine all decompositions $\mathcal{D}=\{(w,x,y,z)\in \mathbb{Z}^4 \mid w^2+x^2+y^2+z^2=4p\}$, see \eqref{sum}
\State discard the elements of $\mathcal{D}$ that produce equivalent matrices, according to  \eqref{eq1}
\For{$(w,x,y,z)\in \mathcal{D}$ }
\State construct $\mathcal{W}$ as the set of back-circulant matrices over $\{\pm1\}$ of order $p$ with row sum $w$
\State similarly, construct $\mathcal{X}, \mathcal{Y}$ for row sums $x$ and $y$, respectively
\For{ $(W,X,Y)\in \mathcal{W}\times \mathcal{X}\times \mathcal{Y}$ satisfying \eqref{eigensum}}
\For{$(a,b)\in \{(\Tri,\Tri),(\Tri,\Inv)\}$}
\State find all $Z$ satisfying \eqref{eq4}--\eqref{eq5}, and construct $H=\HM(W,X,Y,Z)^{a,b}_{1,1,1}$
\If{ $H$ is Hadamard and $H\notin L$ up to equivalence}{ add $H$ to $L$}
\EndIf
\EndFor
\If{$p\equiv 1\bmod 4$ and $wz=xy$ }
\For{$(a,b)\in \{(\Tri,\Tri),(\Tri,\Inv), (\Inv,\Inv)\}$}
\State find all $Z$ satisfying \eqref{eq4}--\eqref{eq5}, and construct $H=\HM(W,X,Y,Z)^{a,b}_{1,1,0}$
\If{ $H$ is Hadamard and $H\notin L$ up to equivalence}{ add $H$ to $L$}
\EndIf
\EndFor
\EndIf
\EndFor
\EndFor
\State \Return $L$.
\end{algorithmic}
\end{algorithm}
}

\subsection{Implementation}\label{secImp}
As a proof of concept, we have implemented our algorithm for the computer algebra system Magma \cite{magma}; for HMs of the orders we have considered $(p\leq 13)$, equivalence of HMs can be tested efficiently via the inbuilt Magma function {\tt IsHadamardEquivalent}.

The classification of CHMs for $p\leq 7$ was first reported by \'{O} Cath\'ain \& R\"oder \cite{CatRod}; our results agree with their classification, and there are $1,3,6$ CHMs for $p=3,5,7$. For $p=11$ and $p=13$, our algorithm found $63$ and $336$ non-equivalent CHMs, respectively. In Table \ref{tabDistro} we give more details for the explicit computations and list the number of non-equivalent CHMs for each group and action. However, note that Theorem \ref{propequivalences} shows that it is sufficient to consider only group $L_{1,1,1}$ and actions $(\Tri,\Tri)$ and $(\Tri,\Inv)$, and group $L_{1,1,0}$ and actions $(\Tri,\Tri)$, $(\Tri,\Inv)$, and $(\Inv,\Inv)$; moreover, group $L_{1,1,0}$ has to be considered only for $p\equiv 1\bmod 4$ and decompositions that satisfy $wz=xy$. Proposition \ref{propequiv} and Theorem \ref{propequivalences} can be used to reduce the number of required computations even further, by considering swapping the entries of the decompositions of $4p$: it follows that in Table \ref{tabDistro} only the numbers highlighted in grey have to be computed. (Algorithm 1 is easily modified to report these values.) Recall that if $p\equiv 3\bmod 4$, then every CHM is equivalent to a WTM or transposed IM. The same holds for $p\equiv 1\bmod 4$ and the group $L_{1,1,0}$ if the block row sums satisfy $wz\ne xy$. Thus, our computations indicate that for $p\equiv 1\bmod 4$, approximately 2/3 of all CHMs are WTM or transposed IM, while 1/3 are not of this form (one CHM for $p=5$ and $102$ CHMs for $p=13$). Since WMs and IMs have enjoyed some attention in the literature, it will be interesting to further investigate the new types of CHMs arising for $p\equiv 1\bmod 4$.  Our computations led to the following conjecture.

\begin{conjecture}
\begin{ithm}
\item There are no CHMs of the form $\HM(W,X,Y,Z)_{1,1,0}^{a,b}$ for $(a,b)\in\{(\Tri,\Tri),(\Tri,\Inv)\}$, respectively. Equivalently,
for circulant blocks $W,X,Y,Z$ there are no CHMs of the form
\begin{equation*}\arraycolsep=1.5pt\def\arraystretch{0.8}
 \left[\begin{array}{cccc}
W & {\color{white}-}X & {\color{white}-}Y & {\color{white}-}Z\\
X & -W & {\color{white}-}Z & -Y\\
Y & {\color{white}-}Z & -W & -X\\
Z & -Y & -X & {\color{white}-}W
\end{array}\right]\text{ and }\
 \left[\arraycolsep=1.5pt\def\arraystretch{0.8}\begin{array}{llll}
W & {\color{white}-}X & {\color{white}-}Y & {\color{white}-}Z\\
X & -W & {\color{white}-}Z & -Y\\
Y^\intercal & {\color{white}-}Z^\intercal & -W^\intercal & -X^\intercal\\
Z^\intercal & -Y^\intercal & -X^\intercal & {\color{white}-}W^\intercal
\end{array}\right].
\end{equation*}
\item For $p\equiv 1\bmod 4$, there exist HMs of the form $\HM(W,X,Y,Z)_{1,1,0}^{\Inv,\Inv}$, and these are not equivalent to WTM or (transposed) IM, yielding a new family of CHMs. Equivalently, for $p\equiv 1\bmod 4$ there exist HMs of the form
\begin{equation*}\arraycolsep=1.5pt\def\arraystretch{0.8}
 \left[\begin{array}{cccc}
W & X & Y & Z\\
X^\intercal & -W^\intercal & Z^\intercal & -Y^\intercal\\
Y^\intercal & Z^\intercal & -W^\intercal & -X^\intercal\\
Z & -Y & -X & W
\end{array}\right]
\end{equation*}
with circulant blocks $W,X,Y,Z$, and these HMs are not equivalent to WTM or (transposed) IM.
\end{ithm}
\end{conjecture}

\begin{table}[!ht]{
\scalebox{0.85}{\hspace{-12pt}\begin{tabular}{c}$ \arraycolsep=2pt
\begin{array}{ r | l | r  r  r  r | r  r  r  r | r  r  r  r | r  r  r  r  ||c}
	  \multicolumn{2}{}{}   & \multicolumn{4}{c|}{L_{1,1,1,}}     &  \multicolumn{4}{c|}{L_{1,0,0}}   &  \multicolumn{4}{c|}{L_{0,1,0}}    & \multicolumn{4}{c||}{L_{1,1,0}} & \#  \\ \hline
	p& $\text{decomp}$  & (\Tri,\Tri) & (\Tri,\Inv) & (\Inv,\Tri) & (\Inv,\Inv) &(\Tri,\Tri) & (\Tri,\Inv) & (\Inv,\Tri) & (\Inv,\Inv) & (\Tri,\Tri) & (\Tri,\Inv) & (\Inv,\Tri) & (\Inv,\Inv) & (\Tri,\Tri) & (\Tri,\Inv) & (\Inv,\Tri) & (\Inv,\Inv)  \\ \hline \hline
	5 &\left[\pm1,1,3,3\right] & \colorbox{lightgray}{1} & \colorbox{lightgray}{2} &  \colorbox{lightgray}{1} & 1 &0 & 0 & 1 & 0 & 0 & 0 & 0 & 0 & \colorbox{lightgray}{0} & \colorbox{lightgray}{0} & 0 & \colorbox{lightgray}{1} \\
	&\left[\pm 1,3,1,3\right] & 1 & 1 & 2 & 1 & 0 & 0 & 0 & 0 & 0 & 1 & 0 & 0 & 0 & \colorbox{lightgray}{0} & 0 & 1 \\
	&\left[\pm 1,3,3,1\right] & 1 & 1 & 1 & 2 & 0 & 0 & 1 & 0 & 0 & 1 & 0 & 0 & 0 & 0 & 0 & 0 \\ \hline
	&\text{non-equiv} & 1 & 2 & 2 & 2 & 0 & 0 & 1 & 0 & 0 & 1 & 0 & 0 & 0 & 0 & 0 & 1 & 3\\ \hline \hline

	7&\left[\pm 1,1,1,5\right] & \colorbox{lightgray}{2} & \colorbox{lightgray}{3} & 3 & 3 & 0 & 0 & 0 & 0 & 0 & 0 & 0 & 0 & 0 & 0 & 0 & 0 \\
	&\left[\pm 1,3,3,3\right] & \colorbox{lightgray}{1} & \colorbox{lightgray}{3} & 3 & 3 & 0 & 0 & 0 & 0 & 0 & 0 & 0 & 0 & 0 & 0 & 0 & 0 \\ \hline
	&\text{non-equiv} & 3 & 6 & 6 & 6 & 0 & 0 & 0 & 0 & 0 & 0 & 0 & 0 & 0 & 0 & 0 & 0 & 6\\ \hline \hline

	11&	\left[\pm 1,5,3,3\right] &  \colorbox{lightgray}{2} &  \colorbox{lightgray}{27} &  \colorbox{lightgray}{38} & 38 & 0 & 0 & 0 & 0 & 0 & 0 & 0 & 0 & 0 & 0 & 0 & 0 \\ 
          &\left[\pm 1,3,3,5\right] & 2 &38 & 38 & 27 & 0 & 0 & 0 & 0 & 0 & 0 & 0 & 0 & 0 & 0 & 0 & 0 \\
	&\left[\pm 1,3,5,3\right] & 2 & 38 & 27 & 38 & 0 & 0 & 0 & 0 & 0 & 0 & 0 & 0 & 0 & 0 & 0 & 0 \\\hline

	&\text{non-equiv} & 2 & 63 & 63 & 63 & 0 & 0 & 0 & 0 & 0 & 0 & 0 & 0 & 0 & 0 & 0 & 0 & 63\\ \hline \hline

	13&\left[\pm 1,1,1,7\right] & \colorbox{lightgray}{2} & \colorbox{lightgray}{74} & 74 & 74 & 0 & 0 & 0 & 0 & 0 & 0 & 0 & 0 & 0 & 0 & 0 & 0 \\
	&\left[\pm 1,1,5,5\right] & \colorbox{lightgray}{2} & \colorbox{lightgray}{42} & \colorbox{lightgray}{52} & 52 & 0 & 0 & 102 & 0 & 0 & 0 & 0 & 0 & \colorbox{lightgray}{0} & \colorbox{lightgray}{0} & 0 & \colorbox{lightgray}{102} \\
	&\left[\pm 1,5,1,5\right] & 2 & 52 & 42 & 52 & 0 & 0 & 0 & 0 & 0 & 102 & 0 & 0 & 0 & \colorbox{lightgray}{0} & 0 & 102 \\
	&\left[\pm 1,5,5,1\right] & 2 & 52 & 52 & 42 & 0 & 0 & 102 & 0 & 0 & 102 & 0 & 0 & 0 & 0 & 0 & 0 \\
	&\left[\pm 3,3,3,5\right] & \colorbox{lightgray}{2} & \colorbox{lightgray}{68} & 68 & 68 & 0 & 0 & 0 & 0 & 0 & 0 & 0 & 0 & 0 & 0 & 0 & 0 \\ \hline
	&\text{non-equiv} & 6 & 234 & 234 & 234 & 0 & 0 & 102 & 0 & 0 & 102 & 0 & 0 & 0 & 0 & 0 & 102 & 336\\ \hline
\end{array}$
\end{tabular}}}
\vspace*{2ex}
\caption{Results of the computations described in Section \ref{secImp}. Entries in rows labelled non-equiv refer to the number of non-equivalent matrices per column. Entries in column labelled \# refer to the total number of CHMs up to equivalence.}\label{tabDistro}
\end{table}

Our computation were performed on a 8-core machine (Intel(R) Core(TM) i7-3770 CPU @ 3.40GHz with 32GB RAM) in combination with Magma V2.23-5.  Calculations for $p>13$ will require a parallelisation of the proposed algorithm and access to a high performance computating facility. For $p=17$, in view of the above conjecture, ignoring WTMs and (transposed) IMs, the only decompositions for $4\cdot 17$ that have to be considered for CHMs $H(W,X,Y,Z)^{a,b}_{1,1,0}$ are  $(3,3,5,5)$ and $(3,5,3,5)$. For $p=19$, all CHMs of order $4p$ are equivalent to WTM or transposed IM.


\end{document}